\documentclass[12pt]{amsart}
\usepackage{amsmath,amsthm,latexsym,amscd,amsbsy,amssymb,amsfonts,fleqn,leqno,
euscript, pb-diagram}

\newtheorem{theorem}{Theorem}[section]

\newtheorem{proposition}[theorem]{Proposition}
\newtheorem{cor}[theorem]{Corollary}
\newtheorem{question}[theorem]{Question}
\theoremstyle{definition}

\theoremstyle{remark}

\numberwithin{equation}{section}

\begin{document}

\title[Maps with dimensionally restricted fibers]
{Maps with dimensionally restricted fibers}

\author[V. Valov]{Vesko  Valov}
\address{Department of Computer Science and Mathematics, Nipissing University,
100 College Drive, P.O. Box 5002, North Bay, ON, P1B 8L7, Canada}
\email{veskov@nipissingu.ca}
\thanks{The author was partially supported by NSERC Grant 261914-08.}



\date{}


\begin{abstract}
We prove that if $f\colon X\to Y$ is a closed surjective map between
metric spaces such that every fiber $f^{-1}(y)$ belongs to a class
of space $\mathrm S$, then there exists an $F_\sigma$-set $A\subset
X$ such that $A\in\mathrm S$ and $\dim f^{-1}(y)\backslash A=0$ for
all $y\in Y$. Here, $\mathrm S$ can be one of the following classes:
(i) $\{M:\mathrm{e-dim}M\leq K\}$ for some $CW$-complex $K$; (ii)
$C$-spaces; (iii) weakly infinite-dimensional spaces. We also
establish that if $\mathrm S=\{M:\dim M\leq n\}$, then $\dim
f\triangle g\leq 0$ for almost all $g\in C(X,\mathbb I^{n+1})$.
\end{abstract}

\maketitle



\section{Introduction}
All spaces in the paper are assumed to be paracompact and all maps
continuous. By $C(X,M)$ we denote all maps from $X$ into $M$. Unless
stated otherwise, all function spaces are endowed with the source
limitation topology provided $M$ is a metric space.

The paper is inspired by the results of Pasynkov \cite{bp1},
Torunczyk \cite{to}, Sternfeld \cite{s} and Levin \cite{lev}.
Pasynkov announced in \cite{bp1} and proved in \cite{bp2} that if
$f\colon X\to Y$ is a surjective map with $\dim f\leq n$, where $X$
and $Y$ are finite-dimensional metric compacta, then $\dim
f\triangle g\leq 0$ for almost all maps $g\in C(X,\mathbb I^n)$ (see
\cite{bp} for a non-compact version of this result). Torunczyk
\cite{to} established (in a more general setting) that if $f$, $X$
and $Y$ are as in Pasynkov's theorem, then for each $0\leq k\leq
n-1$ there exists a $\sigma$-compact subset $A_k\subset X$ such that
$\dim A_k\leq k$ and $\dim f|(X\backslash A_k)\leq n-k-1$.

Next results in this direction were established by Sternfeld and
Levin. Sternfeld \cite{s} proved that if in the cited above results
$Y$ is not-necessarily finite-dimensional, then $\dim f\triangle
g\leq 1$ for almost all $g\in C(X,\mathbb I^n)$ and there exists a
$\sigma$-compact subset $A\subset X$ such that $\dim A\leq n-1$ and
$\dim f|(X\backslash A)\leq 1$. Levin \cite{lev} improved
Sternfeld's results by showing that $\dim f\triangle g\leq 0$ for
almost all $g\in C(X,\mathbb I^{n+1})$, and has shown that this is
equivalent to the existence of an $n$-dimensional $\sigma$-compact
subset $A\subset X$ with $\dim f|(X\backslash A)\leq 0$.

The above results of Pasynkov and Torunczyk were generalized in
\cite{tv} for closed maps between metric space $X$ and $Y$ with $Y$
being a $C$-space (recall that each finite-dimensional paracompact
is a $C$-space \cite{e}). But the question whether the results of
Pasynkov and Torunczyk remain valid without the
finite-dimensionality assumption on $Y$ is still open.

In this paper we provide non-compact analogues of Levin's results
for closed maps between metric spaces.

We say that a topological property of metrizable spaces is an {\em
$\mathrm S$-property} if the following conditions are satisfied: (i)
$\mathrm S$ is hereditary with respect to closed subsets; (ii) if
$X$ is metrizable and $\{H_i\}_{i=1}^\infty$ is a sequence of closed
$\mathrm S$-subsets of $X$, then $\bigcup_{i=1}^\infty H_i\in\mathrm
S$; (iii) a metrizable space $X\in\mathrm S$ provided there exists a
closed surjective map $f\colon X\to Y$ such that $Y$ is a
$0$-dimensional metrizable space and $f^{-1}(y)\in\mathrm S$ for all
$y\in Y$; (iv) any discrete union of $\mathrm S$-spaces is an
$\mathrm S$-space.

Any map whose fibers have a given $\mathrm S$-property is called an
{\em $\mathrm S$-map}.

Here are some examples of $\mathrm S$-properties (we identify
$\mathrm S$ with the class of spaces having the property $\mathrm
S$):
\begin{itemize}
\item $\mathrm S=\{X:\dim X\leq n\}$ for some $n\geq 0$;
\item $\mathrm S=\{X:\dim_G X\leq n\}$, where $G$ is an Abelian group
and $\dim_G$ is the cohomological dimension;
\item more generally, $\mathrm S=\{X:\mathrm{e-dim}
X\leq K\}$, where $K$ is a $CW$-complex and $\mathrm{e-dim}$ is the
extension dimension, see \cite{dr}, \cite{dd};
\item $\mathrm S=\{X: X\hbox{~}\mbox{is weakly infinite-dimensional}\}$;
\item $\mathrm S=\{X: X\hbox{~}\mbox{is a $C$-space}\}$.
\end{itemize}

To show that the property $\mathrm{e-dim}\leq K$ satisfies condition
(iii), we apply \cite[Corollary 2.5]{cv}. For weakly
infinite-dimensional spaces and $C$-spaces this follows from
\cite{hy}.

\begin{theorem}
Let $f\colon X\to Y$ be a closed surjective $\mathrm S$-map with $X$
and $Y$ being metrizable spaces. Then there exists an
$F_\sigma$-subset $A\subset X$ such that $A\in\mathrm S$ and $\dim
f^{-1}(y)\backslash A=0$ for all $y\in Y$. Moreover, if $f$ is a
perfect map, the conclusion remains true provided $\mathrm S$ is a
property satisfying conditions $(i) - (iii)$.
\end{theorem}

Theorem 1.1 was established by Levin \cite[Theorem 1.2]{l2} in the
case $X$ and $Y$ are metric compacta and $\mathrm S$ is the property
$\mathrm{e-dim}\leq K$ for a given $CW$-complex $K$. Levin's proof
of this theorem remains valid for arbitrary $\mathrm S$-property,
but it doesn't work for non-compact spaces.

We say that a map $f\colon X\to Y$ has a countable functional weight
(notation $W(f)\leq\aleph_0$), see \cite{bp}) if there exists a map
$g\colon X\to\mathbb I^{\aleph_0}$ such that $f\triangle g\colon
X\to Y\times\mathbb I^{\aleph_0}$ is an embedding. For example
\cite[Proposition 9.1]{bp2}, $W(f)\leq\aleph_0$ for any closed map
$f\colon X\to Y$ such that $X$ is a metrizable space and every fiber
$f^{-1}(y)$, $y\in Y$, is separable.

\begin{theorem}
Let $X$ and $Y$ be paracompact spaces and $f\colon X\to Y$ a closed
surjective map with $\dim f\leq n$ and $W(f)\leq\aleph_0$. Then
$C(X,\mathbb I^{n+1})$ equipped with the uniform convergence
topology contains a dense subset of maps $g$ such that $\dim
f\triangle g\leq 0$.
\end{theorem}

It was mentioned above that this corollary was established by Levin
\cite[Theorem 1.6]{lev} for metric compacta $X$ and $Y$.  Levin's
arguments don't work for non-compact spaces. We are using  the
Pasynkov's technique  from \cite{bp} to reduced the proof of Theorem
1.2 to the case of $X$ and $Y$ being metric compacta.

Our last results concern the function spaces $C(X,\mathbb I^n)$ and
$C(X,\mathbb I^{\aleph_0})$ equipped with the source limitation
topology. Recall that this topology on $C(X,M)$ with $M$ being a
metrizable space can be described as follows: the neighborhood base
at a given map $h\in C(X,M)$ consists of the sets
$B_\rho(h,\epsilon)=\{g\in C(X,M):\rho(g,h)<\epsilon\}$, where
$\rho$ is a fixed compatible metric on $M$ and $\epsilon:X\to(0,1]$
runs over continuous positive functions on $X$. The symbol
$\rho(h,g)<\epsilon$ means that $\rho(h(x),g(x))<\epsilon(x)$ for
all $x\in X$. It is well known that for paracompact spaces $X$ this
topology doesn't depend on the metric $\rho$ and it has the Baire
property provided $M$ is completely metrizable.

\begin{theorem}
Let $f\colon X\to Y$ be a perfect surjection between paracompact
spaces and $W(f)\leq\aleph_0$.
\begin{itemize}
\item[(i)] The maps $g\in C(X,\mathbb I^{\aleph_0})$ such that
$f\triangle g$ embeds $X$ into $Y\times\mathbb I^{\aleph_0}$ form a
dense $G_\delta$-set in $C(X,\mathbb I^{\aleph_0})$ with respect to
the source limitation topology;
\item[(ii)] If there exists a map $g\in C(X,\mathbb I^n)$ with $\dim
f\triangle g\leq 0$, then all maps having this property form a dense
$G_\delta$-set in $C(X,\mathbb I^n)$ with respect to the source
limitation topology.
\end{itemize}
\end{theorem}

\begin{cor}
Let $f\colon X\to Y$ be a perfect surjection with $\dim f\leq n$ and
$W(f)\leq\aleph_0$, where $X$ and $Y$ are paracompact spaces. Then
all maps $g\in C(X,\mathbb I^{n+1})$ with $\dim f\triangle g\leq 0$
form a dense $G_\delta$-set in $C(X,\mathbb I^{n+1})$ with respect
to the source limitation topology.
\end{cor}

Corollary 1.4 follows directly from Theorem 1.2 and Theorem 1.3(ii).
Corollary 1.5 below follows from Corollary 1.4 and \cite[Corollary
1.1]{bv1}, see Section 3.

\begin{cor}
Let $X$, $Y$ be paracompact spaces and $f\colon X\to Y$ a perfect
surjection with $\dim f\leq n$ and $W(f)\leq\aleph_0$. Then for
every matrizable $ANR$-space $M$ the maps $g\in C(X,\mathbb
I^{n+1}\times M)$ such that $\dim g(f^{-1}(y))\leq n+1$ for all
$y\in Y$ form a dense $G_\delta$-set $E$ in $C(X,\mathbb
I^{n+1}\times M)$ with respect to the source limitation topology.
\end{cor}

Finally, let us formulate the following question concerning property
$\mathrm {S}$ (an affirmative answer of this question yields that
(strong) countable-dimensionality is an $\mathrm {S}$-property):

\begin{question}
Suppose $f\colon X\to Y$ is a perfect surjection between metrizable
spaces such that $\dim Y=0$ and each fiber $f^{-1}(y)$, $y\in Y$, is
$($strongly$)$ countable-dimensional. Is it true that $X$ is
$($strongly$)$ countable-dimensional?
\end{question}

\section{$\mathrm S$-properties and maps into finite-dimensional cubes}

This section contains the proofs of Theorem 1.1 and Theorem 1.2.

\medskip \textit{Proof of Theorem $1.1$}. We follow the proof of
\cite[Proposition 4.1]{vv}. Let us show first that the proof is
reduced to the case $f$ is a perfect map. Indeed, according to
Vainstein's lemma, the boundary $\mathrm{Fr}f^{-1}(y)$ of every
fiber $f^{-1}(y)$ is compact. Defining $F(y)$ to be
$\mathrm{Fr}f^{-1}(y)$ if $\mathrm{Fr}f^{-1}(y)\neq\varnothing$, and
an arbitrary point from $f^{-1}(y)$ otherwise, we obtain a set
$X_0=\bigcup\{F(y):y\in Y\}$ such that $X_0\subset X$ is closed and
the restriction $f|X_0$ is a perfect map. Moreover, each
$f^{-1}(y)\backslash X_0$ is open in $X$ and has the property
$\mathrm S$ (as an $F_\sigma$-subset of the $\mathrm S$-space
$f^{-1}(y)$). Hence, $X\backslash X_0$ being the union of the
discrete family $\{f^{-1}(y)\backslash X_0:y\in Y\}$ of $\mathrm
S$-set is an $\mathrm S$-set. At the same time $X\backslash X_0$ is
open in $X$. Consequently, $X\backslash X_0$ is the union of
countably many closed sets $X_i\subset X$, $i=1,2,..$. Obviously,
each $X_i$, $i\geq 1$, also has the property $\mathrm S$. Therefore,
it suffices to prove Theorem 1.1 for the $\mathrm S$-map
$f|X_0\colon X_0\to Y$.

So, we may suppose that $f$ is perfect. According to \cite{bp},
there exists a map $g\colon X\to\mathbb I^{\aleph_0}$ such that $g$
embeds every fiber $f^{-1}(y)$, $y\in Y$. Let
$g=\triangle_{i=1}^{\infty}g_i$ and $h_i=f\triangle g_i\colon X\to
Y\times\mathbb I$, $i\geq 1$. Moreover, we choose countably many
closed intervals $\mathbb I_j$ such that every open subset of
$\mathbb I$ contains some $\mathbb I_j$. By \cite[Lemma 4.1]{tv},
for every $j$ there exists a $0$-dimensional $F_\sigma$-set
$C_j\subset Y\times\mathbb I_j$ such that $C_j\cap
(\{y\}\times\mathbb I_j)\neq\varnothing$ for every $y\in Y$. Now,
consider the sets $A_{ij}=h_i^{-1}(C_j)$ for all $i,j\geq 1$ and let
$A$ be their union. Since $f$ is an $\mathrm S$-map, so is the map
$h_i$ for any $i$. Hence, $A_{ij}$ has property $\mathrm S$ for all
$i,j$. This implies that $A$ has also the same property.

It remains to show that $\mathrm{dim} f^{-1}(y)\backslash A\leq 0$
for every $y\in Y$. Let $\mathrm{dim} f^{-1}(y_0)\backslash A>0$ for
some $y_0$. Since $g|f^{-1}(y_0)$ is an embedding, there exists an
integer $i$ such that $\mathrm{dim} g_i(f^{-1}(y_0)\backslash A)>0$.
Then $g_i(f^{-1}(y_0)\backslash A)$  has a nonempty interior in
$\mathbb I$. So, $g_i(f^{-1}(y_0)\backslash A)$ contains some
$\mathbb I_j$. Choose $t_0\in\mathbb I_j$ with $c_0=(y_0,t_0)\in
C_j$. Then there exists $x_0\in f^{-1}(y_0)\backslash A$ such that
$g_i(x_0)=t_0$. On the other hand, $x_0\in h_i^{-1}(c_0)\subset
A_{ij}\subset A$, a contradiction. \hfill$\square$

\medskip
\textit{Proof of Theorem $1.2$}. We first prove next proposition
which is a small modification of \cite[Theorem 8.1]{bp}. For any map
$f\colon X\to Y$ we consider the set $C(X,Y\times\mathbb I^{n+1},f)$
consisting of all maps $g\colon X\to Y\times\mathbb I^{n+1}$ such
that $f=\pi_n\circ g$, where $\pi_n\colon Y\times\mathbb I^{n+1}\to
Y$ is the projection onto $Y$. We also consider the other projection
$\varpi_n\colon Y\times\mathbb I^{n+1}\to\mathbb I^{n+1}$. It is
easily seen that the formula $g\rightarrow\varpi_n\circ g$ provides
one-to-one correspondence between $C(X,Y\times\mathbb I^{n+1},f)$
and $C(X,\mathbb I^{n+1})$. So, we may assume that
$C(X,Y\times\mathbb I^{n+1},f)$ is a metric space isometric with
$C(X,\mathbb I^{n+1})$, where $C(X,\mathbb I^{n+1})$ is equipped
with the supremum metric.

\begin{proposition}
Let $f\colon X\to Y$ be an $n$-dimensional surjective map between
compact spaces with $n>0$ and $\lambda\colon X\to Z$ a map into a
metric compactum $Z$. Then the maps $g\in C(X,Y\times\mathbb
I^{n+1},f)$ satisfying the condition below form a dense subset of
$C(X,Y\times\mathbb I^{n+1},f)$: there exists a compact space $H$
and maps $\varphi\colon X\to H$, $h\colon H\to Y\times\mathbb
I^{n+1}$ and $\mu\colon H\to Z$ such that $\lambda=\mu\circ\varphi$,
$g=h\circ\varphi$, $W(h)\leq\aleph_0$ and $\dim h=0$.
\end{proposition}

\begin{proof}
We fix a map $g_0\in C(X,Y\times\mathbb I^{n+1},f)$ and
$\epsilon>0$. Let $g_1=\varpi_n\circ g_0$. Then $\lambda\triangle
g_1\in C(X,Z\times\mathbb I^{n+1})$. Consider also the constant maps
$f'\colon Z\times\mathbb I^{n+1}\to Pt$ and $\eta\colon Y\to Pt$,
where $Pt$ is the one-point space. So, we have $\eta\circ
f=f^{'}\circ (\lambda\triangle g_1)$. According to Pasynkov's
factorization theorem \cite[Theorem 13]{bp3}, there exist metrizable
compacta $K$, $T$ and  maps $f^{*}\colon K\to T$, $\xi_1\colon X\to
K$, $\xi_2\colon K\to Z\times\mathbb I^{n+1}$ and $\eta^{*}\colon
Y\to T$ such that:
\begin{itemize}
\item $\eta^{*}\circ f=f^{*}\circ\xi_1$;
\item $\xi_2\circ\xi_1=\lambda\triangle g_1$;
\item $\dim f^{*}\leq\dim f\leq n$.
\end{itemize}
If $p\colon Z\times\mathbb I^{n+1}\to Z$ and $q\colon Z\times\mathbb
I^{n+1}\to\mathbb I^{n+1}$ denote the corresponding projections, we
have
$$p\circ\xi_2\circ\xi_1=\lambda\hbox{~}\mbox{and}\hbox{~}q\circ\xi_2\circ\xi_1=g_1.$$
Since $\dim f^{*}\leq n$, by Levin's result \cite[Theorem 1.6]{lev},
there exists a map $\phi\colon K\to\mathbb I^{n+1}$ such that $\phi$
is $\epsilon$-close to $q\circ\xi_2$ and $\dim
f^{*}\triangle\phi\leq 0$. Then the map $\phi\circ\xi_1$ is
$\epsilon$-close to $g_1$, so $g=f\triangle(\phi\circ\xi_1)$ is
$\epsilon$-close to $g_0$. Denote $\varphi=f\triangle\xi_1$,
$H=\varphi (X)$ and $h=(id_Y\times\phi)|H$. If $\varpi_H\colon H\to
K$ is the restriction of the projection $Y\times K\to K$ on $H$, we
have
$$\lambda=p\circ\xi_2\circ\xi_1=p\circ\xi_2\circ\varpi_H\circ\varphi,\hbox{~}\mbox{so}\hbox{~}\lambda=\mu\circ\varphi,
\hbox{~}\mbox{where}\hbox{~}\mu=p\circ\xi_2\circ\varpi_H.$$
Moreover, $g=f\triangle(\phi\circ\xi_1)=(id_Y\times\phi)\circ
(f\triangle\xi_1)=h\circ\varphi$. Since $K$ is a metrizable
compactum, $W(\phi)\leq\aleph_0$. Hence, $W(h)\leq\aleph_0$.

To show that $\dim h\leq 0$, it suffices to prove that $\dim
h\leq\dim f^{*}\triangle\phi$. To this end, we show that any fiber
$h^{-1}((y,v))$, where $(y,v)\in Y\times\mathbb I^{n+1}$, is
homeomorphic to a subset of the fiber
$(f^{*}\triangle\phi)^{-1}((\eta^{*}(y),v))$. Indeed, let $\pi_Y$ be
the restriction of the projection $Y\times K\to Y$ on the set $H$.
Since $\eta^{*}\circ f=f^{*}\circ\xi_1$, $H$ is a subset of the
pullback of $Y$ and $K$ with respect to the maps $\eta^{*}$ and
$f^{*}$. Therefore, $\varpi_H$ embeds every fiber $\pi_Y^{-1}(y)$
into $(f^*)^{-1}(y)$, $y\in Y$. Let $a_i=(y_i,k_i)\in H\subset
Y\times K$, $i=1,2$, such that $h(a_1)=h(a_2)$. Then
$(y_1,\phi(k_1))=(y_2,\phi(k_2))$, so $y_1=y_2=y$ and
$\phi(k_1)=\phi(k_2)=v$. This implies $\varpi_H(a_i)=k_i\in
(f^{*})^{-1}\big(\eta^{*}(\pi_Y(a_i))\big)=(f^{*})^{-1}(\eta^{*}(y))$,
$i=1,2$. Hence, $\varpi_H$ embeds the fiber $h^{-1}((y,v))$ into the
fiber $(f^{*}\triangle\phi)^{-1}((\eta^{*}(y),v))$. Consequently,
$\dim h\leq\dim f^{*}\triangle\phi=0$.
\end{proof}

We can prove now Theorem 1.2. It suffices to show every map from
$C(X,Y\times\mathbb I^{n+1},f)$ can be approximated by maps $g\in
C(X,Y\times\mathbb I^{n+1},f)$ with $\dim g\leq 0$. We fix $g_0\in
C(X,Y\times\mathbb I^{n+1},f)$ and $\epsilon>0$. Since
$W(f)\leq\aleph_0$, there exists a map $\lambda\colon X\to\mathbb
I^{\aleph_0}$ such that $f\triangle\lambda$ is an embedding.  Let
$\beta f\colon\beta X\to\beta Y$ be the Cech-Stone extension of the
map $f$. Then $\dim\beta f\leq n$, see \cite[Theorem 15]{bp3}.
Consider also the maps $\beta\lambda\colon\beta X\to\mathbb
I^{\aleph_0}$ and $\bar{g}_0=\beta f\triangle\beta g_1$, where
$g_1=\varpi_n\circ g_0$. According to Proposition 2.1, there exists
a map $\bar{g}\in C(\beta X,\beta Y\times\mathbb I^{n+1},\beta f)$
which is $\epsilon$-close to $\bar{g}_0$ and satisfies the following
conditions: there exists a compact space $H$ and maps $\varphi\colon
\beta X\to H$, $h\colon H\to\beta Y\times\mathbb I^{n+1}$ and
$\mu\colon H\to\mathbb I^{\aleph_0}$ such that
$\beta\lambda=\mu\circ\varphi$, $\bar{g}=h\circ\varphi$,
$W(h)\leq\aleph_0$ and $\dim h=0$. We have the following equalities
$$\beta
f\triangle\beta\lambda=(\pi_n\circ\bar{g})\triangle(\mu\circ\varphi)=(\pi_n\circ
h\circ\varphi)\triangle(\mu\circ\varphi)= \big((\pi_n\circ
h)\triangle\mu\big)\circ\varphi,$$ where $\pi_n$ denotes the
projection $\beta Y\times\mathbb I^{n+1}\to\beta Y$. This implies
that $\varphi$ embeds $X$ into $H$ because $f\triangle\lambda$
embeds $X$ into $Y\times\mathbb I^{\aleph_0}$. Let $g$ be the
restriction of $\bar{g}$ over $X$. Identifying $X$ with $\varphi
(X)$, we obtain that $h$ is an extension of $g$. Hence, $\dim
g\leq\dim h=0$. Observe also that $g$ is $\epsilon$-close to $g_0$,
which completes the proof. \hfill$\square$

\section{Proof of Theorem 1.3 and Corollary 1.5}

\textit{Proof of Theorem $1.3(ii)$.} We first prove condition (ii).
Since $W(f)\leq\aleph_0$, there exists a map $\lambda\colon
X\to\mathbb I^{\aleph_0}$ such that $f\triangle\lambda$ embeds $X$
into $Y\times\mathbb I^{\aleph_0}$. Choose a sequence
$\{\gamma_k\}_{k\geq 1}$ of open covers of $\mathbb I^{\aleph_0}$
with $\mathrm{mesh}(\gamma_k)\leq 1/k$, and let
$\omega_k=\lambda^{-1}(\gamma_k)$ for all $k$. We denote by
$C_{(\omega_k,0)}(X,\mathbb I^{n},f)$ the set of all maps $g\in
C(X,\mathbb I^{n})$ with the following property: every $z\in
(f\triangle g)(X)$ has a neighborhood $V_z$ in $Y\times\mathbb
I^{n}$ such that $(f\triangle g)^{-1}(V_z)$ can be represented as
the union of a disjoint open in $X$ family refining the cover
$\omega_k$. According to \cite[Lemma 2.5]{tv}, each of the sets
$C_{(\omega_k,0)}(X,\mathbb I^{n},f)$, $k\geq 1$, is open in
$C(X,\mathbb I^{n})$ with respect to the source limitation topology.
It follows from the definition of the covers $\omega_k$ that
$\bigcap_{k\geq 1}C_{(\omega_k,0)}(X,\mathbb I^{n},f)$ consists of
maps $g$ with $\dim f\triangle g\leq 0$. Since $C(X,\mathbb I^{n})$
with the source limitation topology has the Baire property, it
remains to show that any $C_{(\omega_k,0)}(X,\mathbb I^{n},f)$ is
dense in $C(X,\mathbb I^{n})$.

To this end, we fix a cover $\omega_m$, a map $g_0\in C(X,\mathbb
I^{n})$ and a function $\epsilon\colon X\to (0,1]$. We are going to
find $h\in C_{(\omega_m,0)}(X,\mathbb I^{n},f)$ such that
$\rho(g_0(x),h(x))<\epsilon(x)$ for all $x\in X$, where $\rho$ is
the Euclidean metric on $\mathbb I^n$. Then, by \cite[Lemma
8.1]{bv}, there exists an open cover $\mathcal U$ of $X$ satisfying
the following condition: if $\alpha\colon X\to K$ is a $\mathcal
U$-map into a paracompact space $K$ (i.e., $\alpha^{-1}(\omega)$
refines $\mathcal U$ for some open cover $\omega$ of $K$), then
there exists a map $q\colon G\to\mathbb I^{n}$, where $G$ is an open
neighborhood of $\overline{\alpha(X)}$ in $K$, such that $g_0$ and
$q\circ\alpha$ are $\epsilon/2$-close with respect to the metric
$\rho$. Let $\mathcal U_1$ be an open cover of $X$ refining both
$\mathcal U$ and $\omega_m$ such that $\inf\{\epsilon(x):x\in U\}>0$
for all $U\in\mathcal U_1$.

Since $\dim f\triangle g\leq 0$ for some $g\in C(X,\mathbb I^{n})$,
according to \cite[Theorem 6]{bv} there exists an open cover
$\mathcal V$ of $Y$ such that for any $\mathcal V$-map $\beta\colon
Y\to L$ into a simplicial complex $L$ we can find a $\mathcal
U_1$-map $\alpha\colon X\to K$ into a simplicial complex $K$ and a
perfect $PL$-map $p\colon K\to L$ with $\beta\circ f=p\circ\alpha$
and $\dim p\leq n$. We can assume that $\mathcal V$ is locally
finite. Take $L$ to be the nerve of the cover $\mathcal V$ and
$\beta\colon Y\to L$ the corresponding natural map. Then there exist
a simplicial complex $K$ and maps $p$ and $\alpha$ satisfying the
above conditions. Hence, the following diagram is commutative.

$$
\begin{array}{cccccccc}
X& \stackrel{\alpha}{\longrightarrow}& K
\\
\vcenter{%
\rlap{$f$}}\;\;\;\Big\downarrow&
 &\Big\downarrow\vcenter{%
\rlap{$p$}} &
\\
Y& \stackrel{\beta}{\longrightarrow}&
 L
\end{array}
$$

Since $K$ is paracompact, the choice of the cover $\mathcal U$
guarantees the existence of a map $\varphi\colon G\to\mathbb I^n $,
where $G\subset K$ is an open neighborhood of
$\overline{\alpha(X)}$, such that $g_0$ and $h_0=\varphi\circ\alpha$
are $\epsilon/2$-close with respect to $\rho$. Replacing the
triangulation of $K$ by a suitable subdivision, we may additionally
assume that no simplex of $K$ meets both $\overline{\alpha(X)}$ and
$K\backslash G$. So, the union $N$ of all simplexes $\sigma\in K$
with $\sigma\cap\overline{\alpha(X)}\neq\varnothing$ is a subcomplex
of $K$ and $N\subset G$. Moreover, since $N$ is closed in $K$,
$p_N=p|N\colon N\to L$ is a perfect map. Therefore, we have the
following commutative diagram:

\begin{picture}(120,95)(-100,0)
\put(30,10){$L$} \put(0,30){$Y$} \put(12,28){\vector(3,-2){18}}
\put(14,14){\small $\beta$} \put(1,70){$X$}
\put(5,66){\vector(0,-1){25}} \put(-1,53){\small $f$}
\put(11,73){\vector(1,0){45}} \put(30,77){\small $h_0$}
\put(12,68){\vector(3,-2){18}} \put(15,56){\small $\alpha$}
\put(31,50){$N$} \put(35,46){\vector(0,-1){25}}
 \put(37,33){\small $p_N$}
\put(46,58){\vector(4,3){13}} \put(44,64){\small $\varphi$}
 \put(60,70){$\mathbb I^n$}
\end{picture}

Using that $\alpha$ is a $\mathcal U_1$-map and
$\inf\{\epsilon(x):x\in U\}>0$ for all $U\in\mathcal U_1$, we can
construct a continuous function $\epsilon_1:N\to(0,1]$ and an open
cover $\gamma$ of $N$ such that $\epsilon_1\circ\alpha\leq\epsilon$
and $\alpha^{-1}(\gamma)$ refines $\mathcal U_1$. Since $\dim
p_N\leq\dim p\leq n$ and $L$, being a simplicial complex, is a
$C$-space, we can apply \cite[Theorem 2.2]{tv} to find a map
$\varphi_1\in C_{(\gamma,0)}(N,\mathbb I^n,p_N)$ which is
$\epsilon_1/2$-close to $\varphi$. Let $h=\varphi_1\circ\alpha$.
Then $h$ and $h_0$ are $\epsilon/2$-close because
$\epsilon_1\circ\alpha\leq\epsilon$. On the other hand, $h_0$ is
$\epsilon/2$-close to $g_0$. Hence, $g_0$ and $h$ are
$\epsilon$-close.

It remains to show that $h\in C_{(\omega_m,0)}(X,\mathbb I^{n},f)$.
To this end, fix  a point $z=(f(x),h(x))\in (f\triangle h)(X)\subset
Y\times\mathbb I^n$ and let $y=f(x)$. Then
$w=(p_N\triangle\varphi_1)(\alpha(x))=(\beta(y),h(x))$. Since
$\varphi_1\in C_{(\gamma,0)}(N,\mathbb I^n,p_N)$, there exists a
neighborhood $V_w$ of $w$ in $L\times\mathbb I^n$ such that
$W=(p_N\triangle\varphi_1)^{-1}(V_w)$ is a union of a disjoint open
family in $N$ refining $\gamma$. We can assume that
$V_w=V_{\beta(y)}\times V_{h(x)}$, where $V_{\beta(y)}$ and
$V_{h(x)}$ are neighborhoods of $\beta(y)$ and $h(x)$ in $Y$ and
$\mathbb I^n$, respectively. Consequently, $(f\triangle
h)^{-1}(\Gamma)=\alpha^{-1}(W)$, where
$\Gamma=\beta^{-1}\big(V_{\beta(y)}\big)\times V_{h(x)}$. Finally,
observe that $\alpha^{-1}(W)$ is a disjoint union of an open in $X$
family refining $\omega_m$. Therefore, $h\in
C_{(\omega_m,0)}(X,\mathbb I^{n},f)$. \hfill$\square$

\textit{Proof of Theorem $1.3(i)$.} Let $\lambda$ and $\omega_k$ be
as in the proof of Theorem 1.3(i). Denote by $C_{\omega_k}(X,\mathbb
I^{\aleph_0},f)$ the set of all $g\in C(X,\mathbb I^{\aleph_0})$
such that $f\triangle g$ is an $\omega_k$-map. It can be shown that
every $C_{\omega_k}(X,\mathbb I^{\aleph_0},f)$ is open in
$C(X,\mathbb I^{\aleph_0})$ with the source limitation topology (see
\cite[Proposition 3.1]{tv1}). Moreover, $\bigcap_{k\geq
1}C_{\omega_k}(X,\mathbb I^{\aleph_0},f)$ consists of maps $g$ with
$f\triangle g$ embedding $X$ into $Y\times\mathbb I^{\aleph_0}$. So,
we need to show that each $C_{\omega_k}(X,\mathbb I^{\aleph_0},f)$
is dense in $C(X,\mathbb I^{\aleph_0})$ equipped with the source
limitation topology.

To prove this fact we follow the notations and the arguments from
the proof of Theorem 1.3(ii) (that $C_{(\omega_k,0)}(X,\mathbb
I^{n},f)$ are dense in $C(X,\mathbb I^{n})$) by considering $\mathbb
I^{\aleph_0}$ instead of $\mathbb I^n$. We fix a cover $\omega_m$, a
map $g_0\in C(X,\mathbb I^{\aleph_0})$ and a function $\epsilon\in
C(X,(0,1])$. Since $W(f)\leq\aleph_0$, we can apply Theorem 6 from
\cite{bv} to find an open cover $\mathcal V$ of $Y$ such that for
any $\mathcal V$-map $\beta\colon Y\to L$ into a simplicial complex
$L$ there exists a $\mathcal U_1$-map $\alpha\colon X\to K$ into a
simplicial complex $K$ and a perfect $PL$-map $p\colon K\to L$ with
$\beta\circ f=p\circ\alpha$. Proceeding as before, we find a map
$h=\varphi_1\circ\alpha$ which is $\epsilon$-close to $g_0$, where
$\varphi_1\in C_{\gamma}(N,\mathbb I^{\aleph_0},p_N)$. It is easily
seen that $\varphi_1\in C_{\gamma}(N,\mathbb I^{\aleph_0},p_N)$
implies $h\in C_{\omega_m}(X,\mathbb I^{\aleph_0},f)$. So,
$C_{\omega_m}(X,\mathbb I^{\aleph_0},f)$ is dense in $C(X,\mathbb
I^{\aleph_0})$. \hfill$\square$

\textit{Proof of Corollary $1.5$.} It follows from \cite[Proposition
2.1]{bv1} that the set $E$ is $G_\delta$ in $C(X,\mathbb
I^{n+1}\times M)$. So, we need to show it is dense in $C(X,\mathbb
I^{n+1}\times M)$. To this end, we fix $g^0=(g_1^0,g_2^0)\in
C(X,\mathbb I^{n+1}\times M)$ with $g_1^0\in C(X,\mathbb I^{n+1})$
and $g_2^0\in C(X,M)$. Since, by Corollary 1.4, the set
$$G_1=\{g_1\in C(X,\mathbb I^{n+1}):\dim f\triangle g_1\leq 0\}$$ is
dense in $C(X,\mathbb I^{n+1})$, we may approximate $g_1^0$ by a map
$h_1\in G_1$. Then, according to \cite[Corollary 1.1]{bv1}, the maps
$g_2\in C(X,M)$ with $\dim g_2\big((f\triangle h_1)^{-1}(z)\big)=0$
for all $z\in Y\times\mathbb I^{n+1}$ form a dense subset $G_2$ of
$C(X,M)$. So, we can approximate $g_2^0$ by a map $h_2\in G_2$. Let
us show that the map $h=(h_1,h_2)\in C(X,\mathbb I^{n+1})\times M$
belongs to $E$. We define the map $\pi_h\colon (f\triangle h)(X)\to
(f\triangle h_1)(X)$,
$\pi_h\big(f(x),h_1(x),h_2(x)\big)=\big(f(x),h_1(x)\big)$, $x\in X$.
Because $f$ is perfect, so is $\pi_h$. Moreover,
$(\pi_h)^{-1}\big(f(x),h_1(x)\big)=h_2\big(f^{-1}(f(x))\cap
h_1^{-1}(h_1(x))\big)$, $x\in X$. So, every fiber of $\pi_h$ is
$0$-dimensional. We also observe that
$\pi_h\big(h(f^{-1}(y))\big)=(f\triangle h_1)(f^{-1}(y))$ and the
restriction $\pi_h|h(f^{-1}(y))$ is a perfect surjection between the
compact spaces $h(f^{-1}(y))$ and $(f\triangle h_1)(f^{-1}(y))$ for
any $y\in Y$. Since $(f\triangle
h_1)(f^{-1}(y))\subset\{y\}\times\mathbb I^{n+1}$, $\dim (f\triangle
h_1)(f^{-1}(y))\leq n+1$, $y\in Y$. Consequently, applying the
Hurewicz's dimension-lowering theorem \cite{e} for the  map
$\pi_h|h(f^{-1}(y))$, we have $\dim h(f^{-1}(y))\leq n+1$.
Therefore, $h\in E$, which completes the proof. \hfill$\square$



\bibliographystyle{amsplain}

\end{document}